\documentclass[12pt,a4paper]{amsart}
\usepackage{amsfonts,amsmath,amsthm,amssymb,amscd}
\usepackage{setspace}
\begin{document}
\numberwithin{equation}{section}
\newtheorem{teo}{Theorem}[section]
\newtheorem{lem}[teo]{Lemma}
\newtheorem{pro}[teo]{Proposition}
\newtheorem{defi}[teo]{Definition}
\newtheorem{cor}[teo]{Corollary}
\newtheorem{oss}[teo]{Remark}
\newtheorem{exa}[teo]{Example}
\def\a{\alpha}
\def\b{\beta}
\def\g{\gamma}
\def\l{\lambda}
\def\d{\delta}
\def\e{\epsilon}
\def\S{{\mathbb{S}}}
\def\N{{\mathbb{N}}}
\def\G{{\Gamma}}
\def\Z{{\Sigma}}
\def\O{{\Omega}}
\def\ds{\displaystyle}
\newcommand{\R}{\mathbb{R}}
\def\dim{{\bf Proof. }}
\def \fine{\hfill {\rule {2mm} {2mm}} \vspace{3mm}}

\newcommand{\rn}{\mathbb{R}^{n}}
\newcommand{\C}{\mathbb{C}}
\newcommand{\rl}{\mathbb{R}^{l}}
\newcommand{\Ker}{{\rm Ker}}
\renewcommand{\Im}{{\rm Im}}
\renewcommand{\Re}{{\rm Re}}
\newcommand{\xp}{x^{\prime}}
\newcommand{\xpp}{x^{\prime\prime}}
\newcommand{\xppp}{x^{\prime\prime\prime}}
\newcommand{\xip}{{\xi^{\prime}}}
\newcommand{\xipp}{\xi^{\prime\prime}}
\newcommand{\xippp}{\xi^{\prime\prime\prime}}
\newcommand{\bp}{\beta^{\prime}}
\newcommand{\bpp}{\beta^{\prime\prime}}
\newcommand{\ol}{\overline{L}}
\newcommand{\oo}{\overline{\omega}}
\newcommand{\zp}{z^\prime}
\newcommand{\zpp}{z^{\prime\prime}}
\newcommand{\zppp}{z^{\prime\prime\prime}}
\newcommand{\tp}{t^\prime}
\newcommand{\tpp}{t^{\prime\prime}}
\newcommand{\tppp}{t^{\prime\prime\prime}}
\newcommand{\ozp}{\overline{z}^\prime}
\newcommand{\ozpp}{\overline{z}^{\prime\prime}}
\newcommand{\ozppp}{\overline{z}^{\prime\prime\prime}}
\newcommand{\xl}{x^\lambda}
\newcommand{\xil}{\xi^\lambda}
\newcommand{\cl}{\chi_\lambda}
\newcommand{\pl}{p_\lambda}
\newcommand{\al}{a_\lambda}
\renewcommand{\l}{\lambda}

\title[Restriction conjecture]{
Slicing surfaces and Fourier
restriction conjecture}
\author{Fabio Nicola}
\address{Dipartimento di Matematica, Politecnico di Torino,
Corso Duca degli Abruzzi,
 24 - 10129 Torino, Italy}
 \email{fabio.nicola@polito.it
}
\thanks{}
\begin{abstract}
We deal with the restriction
phenomenon for the Fourier
transform. We prove that each
of the restriction
conjectures for the sphere,
the paraboloid, the elliptic
hyperboloid in $\R^n$ implies
that for the cone in
$\R^{n+1}$. We also prove a
new restriction estimate for
any surface in $\R^3$ locally
isometric to the plane and of
finite type.
\end{abstract} {}\smallskip
 \maketitle

\section{Introduction and discussion of
 the results}\label{intro}
 Let $S$
be a smooth hypersurface with
(possibly empty) boundary in
$\R^n$, $n\geq 2$, or a
compact subset (with
non-empty interior) of such a
hypersurface, and let
$d\sigma$ be the surface
measure on $S$. Denote by
$\hat{f}$ the Fourier
transform of the function
$f$. We deal with the
so-called restriction
estimate
\[
\|\hat{f}|_S\|_{L^q(S,d\sigma)}\leq
C_{p,q,S}\|f\|_{L^p(\R^n)},
\]
for all Schwartz functions
$f$, or equivalently with the
extension estimate
$$
\|(ud\sigma)^\vee\|_{L^{p'}(\R^n)}\leq
C_{p,q,S}\|u\|_{L^{q'}(S,d\sigma)},
$$ for all smooth functions
$u$ with compact support in
$S$. We denote these
estimates by $R_S(p\to q)$
and $R^\ast_S(q'\to p')$
respectively. Here
$(ud\sigma)^\vee$ is the
inverse Fourier transform of
the measure $ud\sigma$.\\ We
will mostly be interested in
the case in which $S$ is the
sphere, or the elliptic
paraboloid, or the elliptic
hyperboloid in $\R^n$ and
also the light cone in
$\R^{n+1}$. Precisely we
define
\[S_{sphere}=\{\xi\in\R^n:
|\xi|=1\},\quad
S_{parab}=\{(\xi',\xi_n)\in\R^n:\xi_n=
\frac{1}{2}|\xi'|^2\},
\]
\[
S_{hyperb}=\{(\xi',\xi_n)\in\R^n:\xi_n=
\sqrt{1+|\xi'|^2}\},\]
\[
S_{cone}=\{(\xi,\tau)\in\R^{n+1}:\tau=|\xi|,\
\xi\not=0\}.\] On the sphere
we take the usual surface
measure, whereas the measure
on the paraboloid and on the
hyperboloid are defined as
the pull-back under the
projection
$(\xi',\xi_n)\mapsto\xi'$ of
the measures $d\xi'$ and
$\frac{d\xi'}{\sqrt{1+|\xi'|^2}}$
respectively. Also, on the
cone we consider the Lorentz
invariant measure
$d\sigma_{cone}$ defined as
the pull-back under the
projection
$(\xi,\tau)\mapsto\xi$ of the
measure $\frac{d\xi}{|\xi|}$.
Of course for the cone we
mean that, in the estimates
$R_S(p\to q)$ and
$R^\ast(q'\to p')$, $n$ must
be replaced by $n+1$.\par The
restriction estimate is
conjectured to hold in the
following cases (see
\cite{tomas,strichartz,stein,wolff2}
and especially \cite{tao}).
\par\medskip\noindent
{\bf Restriction
conjecture}.\par
\begin{itemize}{\it
\item[{\it i})] Suppose that
\begin{equation}\label{f01}
\frac{p'}{n+1}\geq\frac{q}{n-1},\quad
p'>\frac{2n}{n-1}.
\end{equation}
Then $R_S(p\to q)$ holds when
$S$ is any compact subset of
a hypersurface in $\R^n$ with
non-vanishing Gaussian
curvature or of the cone in
$\R^{n+1}$.\item[{\it ii})]
Suppose that
\begin{equation}\label{f02}
\frac{p'}{n+1}=\frac{q}{n-1},\quad
p'>\frac{2n}{n-1},
\end{equation}
i.e., the pair $p,q$ is scale
invariant. Then $R_S(p\to q)$
holds when $S$ the whole
paraboloid in $\R^n$, or the
whole hyperboloid in $\R^n$,
or the whole cone in
$\R^{n+1}$.} \end{itemize}
The conditions \eqref{f01}
and \eqref{f02} are known to
be necessary.\par This
fascinating conjecture was
proved for curves in the
plane by Zygmund
\cite{zigmund} and Fefferman
\cite{fefferman}, for the
cone in $\R^3$ by Barcel\'o
\cite{barcelo} and for the
cone in $\R^4$ by Wolff
\cite{wolff}. Partial results
in higher dimensions have
been obtained by
 many authors, culminating in the work of
 Wolff \cite{wolff}, who
 proved {\it i}) for the cone under
 the additional condition
$p'>\frac{2(n+3)}{n+1}$, and
Tao \cite{tao0}, who proved
{\it i}) for the sphere and
the paraboloid under the
additional condition
$p'>\frac{2(n+2)}{n}$.
\par
For the paraboloid any scale
invariant result for compact
subsets extends automatically
to the whole paraboloid,
whereas, for the cone, this
is not so immediate. The
above mentioned results by
Wolff should however extend
(in the scale invariant case)
to the whole cone (possibly
after conceding an epsilon in
the exponents) by using the
techniques in
\cite{tao01,tao02}
(\cite{taop}). Moreover, it
is well known (see  e.g.
\cite{tao02} and Problem 1.1
of \cite{tao}) that the
restriction conjecture for
the sphere (or any other
hypersurface with $n$
positive principal
curvatures) implies the
restriction conjecture for
the paraboloid. See
\cite{tao} for a detailed
survey.
\par As one sees, the numerology for the
sphere, the paraboloid and the
hyperboloid in $\R^n$ agrees with that
of the cone in $\R^{n+1}$.
Heuristically this is explained by
observing that the cone has one extra
dimension which nevertheless, being
flat, is not expected to produce any
contribution in the restriction
estimate. A deeper investigation of
this link, in the case of the
paraboloid and the cone, has been also
exploited in \cite{tao03} (Proposition
17.5; see also the discussion in
\cite{tao}, pages 7-8). However, to our
knowledge, proofs of formal
implications did not appear in the
literature and do not seem known.
\par The first result of this note shows
that each of the restriction
conjectures for the sphere,
the paraboloid, and the
hyperboloid in $\R^n$ implies
the restriction conjecture
for the cone in $\R^{n+1}$.
More precisely, we have the
following result.
\begin{teo}\label{t1} a)
Assume, {\rm (}for some $p,q$
satisfying \eqref{f01}{\rm )}, one of
the following hypotheses:
\begin{itemize}
\item[{\it i})] $R_{S}(p\to
q)$ holds for the sphere $S$
in $\R^n$;\item[{\it ii})]
$R_{S}(p\to q)$ holds for
every compact subset $S$ of
the paraboloid in
$\R^n$;\item[{\it iii})]
$R_{S}(p\to q)$ holds for
every compact
 subset $S$ of the hyperboloid in
 $\R^n$.
\end{itemize}
Then $R_S(\ p\to q)$ holds for every
compact subset $S$ of the cone in
$\R^{n+1}$.\par\medskip\noindent b)
Assume one of the following
hypotheses:\par
\begin{itemize}
\item[{\it i})'] $R_{S}(p\to
q)$ holds, for some $p,q$ satisfying
\eqref{f02}, for the sphere $S$ in
$\R^n$;
\item[{\it ii})'] $R_{S}(p\to
q)$ holds, for some $p,q$ satisfying
\eqref{f02}, for the whole paraboloid
$S$ in $\R^n$;\par \item[{\it iii})']
$R_{S}(p\to q)$, for some $p,q$
satisfying \eqref{f02}, holds for the
whole hyperboloid $S$ in
 $\R^n$.
\end{itemize}
Then $R_S(p\to q)$ holds for
the whole cone $S$ in
$\R^{n+1}$.
\end{teo}
Notice that, even in the case
of the whole cone, there is
no loss in the exponents.
Also, by combining Theorem
\ref{t1} with the sharp
restriction theorem for the
circle in the plane, we
obtain another proof of the
sharp restriction theorem for
the cone in $\R^3$.
\par As might be expected, the proof
exploits the fact that the
sphere, the paraboloid and
the hyperboloid are sections
of the light cone. We observe
that a similar trick goes
back to H\"ormander's paper
\cite{hormander}, and has
been used, in different
contexts, by Vega
\cite{vega}, Carbery
\cite{carbery}, Mockenhaupt,
Seeger and Sogge
\cite{mockenhaupt}, Tataru
\cite{tataru} (Appendix B),
 Tao \cite{tao05} (Proposition 4.1), Burq,
  Lebeau
and Planchon \cite{burq}
(Theorem 2). Also, this trick
works for more general
surfaces of revolution than
the cone, e.g. the
paraboloid, but one
 no longer obtains sharp
estimates.\par The second
result of this note consists
in a restriction theorem for
surfaces in $\R^3$ with
Gaussian curvature vanishing
everywhere.\par Here we say
that a point $P_0$ of a
hypersurface $S$ is of type
$k$ if $k$ is the order of
contact of $S_0$ with its
tangent plane at $P_0$
(\footnote{In general, if $S$
is a smooth $m$-dimensional
submanifold of $\R^n$, $1\leq
m\leq n-1$, and
$\phi:U\subset \R^m\to\R^n$
is a local parametrisation of
$S$, with $\phi(x_0)=P_0$,
$P_0$ is called of type $k$
if $k$ is the smallest
integer such that, for each
unit vector $\eta$, there
exists an $\alpha$ with
$|\alpha|\leq k$ for which
$\partial^\alpha
[\phi(x)\cdot\eta]|_{x=x_0}\not=0$.
See \cite{stein}, page 350
for more details.}).
\begin{teo}\label{t2}
Let $S$ be a surface in
$\R^3$ and $P_0\in S$ a point
of type $k$. Suppose that $S$
has Gaussian curvature
vanishing identically near
$P_0$. Then there exists a
compact neighbourhood
$S_0\subset S$ of $P_0$ and a
constant $C$ such that
\begin{equation}\label{pm}
\|(ud\sigma)^\vee\|_{L^{p'}(\R^3)}\leq
C\|u\|_{L^{q'}(S,d\sigma)},
\end{equation}
for every $p'>4$, $p'\geq
k+2$, $p'\geq (k+1)q$, and
every smooth $u$ supported in
$S_0$.
\end{teo}
Notice that the numerology
agrees with the sharp
restriction theorem by Sogge
\cite{sogge} for curves of
finite type in the plane,
according to the fact that
$S$ has one principal
curvature identically zero
near $P_0$. We also recall
that the hypothesis of the
vanishing of the Gaussian
curvature near $P_0$ is
equivalent to saying that a
convenient neighbourhood of
$P_0$ is isometric to the
plane. Classical examples of
surfaces with such property
are given by the developable
surfaces (hence, cones,
cylinders, tangent
developables); see Spivak
\cite{spivak} for details. We
point out that, when instead
the Gaussian curvature
vanishes on a one dimensional
submanifold and there are no
umbilic points, decay
estimates for $(d\sigma)^\vee
$ have been recently obtained
by Erd\H{o}s and Salmhofer
\cite{erdos}.
\par The proof of Theorem
\ref{t2} uses the simple
idea, as above, of
transferring restriction
estimates from slices of a
surface (given here by the
above mentioned Sogge's
result) to the surface
itself. However, to this end
we need to prove a new normal
form for $S$ near $P_0$
(Proposition \ref{nf} below),
defined in terms of an {\it
orthogonal} transformation in
$\R^3$, hence particularly
convenient for the
restriction problem. We refer
to Remark
\ref{osservazioneaggiunta}
below for a comparison with
the normal forms in Schulz
\cite{schulz}.\par
 The next two
sections are devoted to the proof of
Theorems \ref{t1} and \ref{t2}
respectively.\par\medskip\noindent{\it
 Acknowledgments.}  I would like to express my gratitude
 to Professor Terence Tao for correspondence
  on the
 subject of this paper, and
 to Professors Antonio J. Di Scala
 and Luigi Rodino
 for helpful discussions.
 Also, I wish to thank the
 referee for interesting comments.
\section{Proof of Theorem \ref{t1}}
We first fix the notation and
recall some preliminary
results which are needed in
the proof of Theorem
\ref{t1}.\par Given a measure
space
$X=(X,\mathcal{B}_X,\mu_X)$,
we denote by
$L^{\alpha,\beta}=L^{\alpha,\beta}(X)$,
$0<\alpha<\infty$, $0<q\leq
\infty$, the Lorentz spaces
on $X$. Hence,
$\|f\|_{L^{\alpha,\beta}}=
\|\lambda\mu(\{|f|\geq\lambda\})^{1/\alpha}\|_{L^\beta\left(\R^+,\frac{d\lambda}
{\lambda}\right)}$. We recall
(see e.g. \cite{steinweiss})
that
$L^{\alpha,\alpha}=L^\alpha$,
and
$L^{\alpha,\beta_1}\hookrightarrow
L^{\alpha,\beta_2}$ if
$\beta_1\leq\beta_2$.
Moreover, H\"older's
inequality for Lorentz spaces
reads as follows: if
$0<\alpha_1,\alpha_2,\alpha<\infty$
and
$0<\beta_1,\beta_2,\beta\leq\infty$
obey
$\frac{1}{\alpha}=\frac{1}{\alpha_1}+
\frac{1}{\alpha_2}$ and
$\frac{1}{\beta}=\frac{1}{\beta_1}+
\frac{1}{\beta_2}$ then
\begin{equation}\label{holder}
\|fg\|_{L^{\alpha,\beta}}\leq
\|f\|_{L^{\alpha_1,\beta_1}}\|f\|_{L^{\alpha_2,\beta_2}}.
\end{equation}
We also recall that there is a sharp
version of the Hausdorff-Young
inequality in terms of Lorentz spaces
in $\R^n$, with the Lebesgue measure
(Corollary 3.16 in \cite{steinweiss},
page 200). Namely, if $1< p\leq 2$, we
have \footnote{We write $A\lesssim B$
if $A\leq C B$ for some constant $C>0$
which may depend on parameters, like
Lebesgue exponents or the dimension
$n$.}
\begin{equation}\label{au}
\|\hat{u}\|_{L^{p'}}\lesssim
\|u\|_{L^{p,p'}}.
\end{equation}
We will need the following
lemma on the interchange of
norms. Consider two measure
spaces
$X=(X,\mathcal{B}_X,\mu_X)$
and
$Y=(Y,\mathcal{B}_Y,\mu_Y)$
and a function $f(x,y)$ on
the product space $(X\times
Y)=(X\times Y,
\mathcal{B}_X\times\mathcal{B}_Y,\mu_X
\times\mu_Y)$. We define the
mixed norms
$L^{\alpha,\beta}_xL^{\gamma,\delta}_y(X\times
Y)$ of $f$ as
\[
\|f\|_{L^{\alpha,\beta}_x
L^{\gamma,\delta}_y(X\times
Y)}=\|\|f(x,\cdot)\|_{L^{\gamma,\delta}(Y)}
\|_{L^{\alpha,\beta}(X)}.
\]
\begin{lem} If
$1< p\leq 2$ we have
\begin{equation}\label{inter}
\|u\|_{L^{p'}_x
L^{p,p'}_y}\lesssim\|u\|_{L^{p,p'}_y
L^{p'}_x}.
\end{equation}
\end{lem}
\begin{proof}
By Minkowski's  inequality we
have
\begin{equation}\label{a100}
\|u\|_{L^{p'}_x
L^1_y}\leq\|u\|_{L^1_y
L^{p'}_x}.
\end{equation}
So the desired estimate
follows by real interpolation
from \eqref{a100} and the
trivial estimate
$\|u\|_{L^{p'}_x
L^{p'}_y}=\|u\|_{L^{p'}_y
L^{p'}_x}$. Indeed, if
$\frac{1}{p}=\frac{1-\theta}{1}+\frac{\theta}{p'}$
we have
\[
[L^{p'}_x L^1_y,L^{p'}_x
L^{p'}_y]_{\theta,p'}
=L^{p'}_x L^{p,p'}_y
\]
 by (3) of \cite{triebel},
page 128, and (16) of
\cite{triebel}, page 134,
whereas
\[
[L^{1}_y L^{p'}_x,L^{p'}_y
L^{p'}_x]_{\theta,p'}
=L^{p,p'}_y L^{p'}_x,
\]
again by (16) of
\cite{triebel}, page 134.\par
This concludes the proof.
\end{proof}
We now prove Theorem \ref{t1}
separately in the three
cases, namely for the sphere,
the paraboloid and the
hyperboloid. Although the
three proofs follow a similar
pattern and, in fact, the
part b) for the sphere and
the hyperboloid follow from
that for the paraboloid
combined with results in
\cite{tao02}, for the
convenience of the reader we
present each proof in a
self-contained form. With
abuse of notation we always
identify functions on the
cone with functions in
$\R^n_{\xi'}$. Moreover we
will make use, both in the
hypotheses and in the
conclusion, of the
formulation $R^\ast_S(q'\to
p')$, which is easily seen to
be equivalent to $R_S(p\to
q)$. We also suppose $p>1$,
since the case $p=1$ is
trivial.
\begin{proof}[Proof that spherical restriction $\Rightarrow$ conical restriction]
Here we prove the conclusions of
Theorem \ref{t1},  under {\it i}) or
{\it i})'. We use polar coordinates
$(r,\omega)$ in $\R^n_\xi$, and we
denote by $d\omega$ the measure on the
sphere and by $L^{\alpha,\beta}$ the
Lorentz spaces on $\R^+$ with respect
to Lebesgue measure.
\par
Assume {\it i})'. We have
\begin{align}
(u\,d\sigma_{cone})^\vee(x,t)&=\int
e^{2\pi i(x\xi+t|\xi|)}
u(\xi)\,\frac{d\xi}{|\xi|}\nonumber\\
&=\int_0^{+\infty} e^{2\pi itr}
r^{n-2}\int_{\mathbb{S}^{n-1}}
 e^{2\pi irx\omega}
 u(r\omega)\,d\omega\,dr.\label{hh1}
\end{align}
Hence, by the Hausdorff-Young
inequality \eqref{au} and
\eqref{inter} we have
\begin{align}
\|(u\,d\sigma_{cone})^\vee\|_{L^{p'}}&=
\|(u\,d\sigma_{cone})^\vee\|_{L^{p'}_x
L^{p'}_t}\nonumber\\
&\lesssim\|r^{n-2}\int_{\mathbb{S}^{n-1}}
 e^{2\pi irx\omega}
 u(r\omega)\,d\omega\|_{{L^{p'}_x}L^{p,p'}_r}
 \nonumber\\
&\lesssim\|r^{n-2}\int_{\mathbb{S}^{n-1}}
 e^{2\pi irx\omega}
 u(r\omega)\,d\omega\|_{{L^{p,p'}_r}L^{p'}_x}.\nonumber
\end{align}
Now, a change of variables and the
hypothesis give
\begin{align}
\|\int_{\mathbb{S}^{n-1}}
 e^{2\pi irx\omega}
 u(r\omega)\,d\omega\|_{L^{p'}_x}&=
 r^{-\frac{n}{p'}}\|\int_{\mathbb{S}^{n-1}}
 e^{2\pi ix\omega}
 u(r\omega)\,d\omega\|_{L^{p'}_x}\nonumber\\
 &\lesssim
 r^{-\frac{n}{p'}}\|u(r\cdot)
 \|_{L^{q'}(\mathbb{S}^{n-1})}.
 \nonumber
 \end{align}
 Hence we obtain
 \begin{align}
 \|(u\,d\sigma_{cone})^\vee\|_{L^{p'}}
 &\lesssim
 \| r^{n-2-\frac{n}{p'}}\|u(r\cdot)\|_{L^{q'}
 (\mathbb{S}^{n-1})}
 \|_{L^{p,p'}}\label{g0}\\
 &=\|r^{\frac{n-2}{q}-\frac{n}{p'}}
 \cdot r^\frac{n-2}{q'}\|u(r\cdot)\|_
 {L^{q'}(\mathbb{S}^{n-1})}\|_{L^{p,p'}}\nonumber\\
 &\lesssim
 \|\underbrace{r^{\frac{n-2}{q}-\frac{n}{p'}}}_{F(r)}
 \cdot \underbrace{r^\frac{n-2}{q'}\|u(r\cdot)\|_
 {L^{q'}(\mathbb{S}^{n-1})}}_{G(r)}\|_{L^{p,q'}},
 \nonumber
 \end{align}
where the last inequality
follows because $p'>q'$. Now,
let $\alpha$ be defined by
 $\frac{1}{\alpha}+\frac{1}{q'}=\frac{1}{p}$.
 A direct computation shows that $F\in
 L^{\alpha,\infty}$ (since the pair $p,q$ is
 scale invariant).
Hence, by
 H\"older's inequality for
 Lorentz spaces \eqref{holder}, we see that the last
 expression is not greater than
 \[
 \|F\|_{L^{\alpha,\infty}}\|G\|_{L^{q',q'}}=
 C\|G\|_{L^{q'}}
 =C\|u\|_{L^{q'}(\R^n,\frac{d\xi}{|\xi|})}.
 \]
 This concludes the proof of the restriction estimates for the whole cone.
\par
The proof of the restriction
estimate for compact subsets
of the cone, under the
assumption {\it i}), is even
easier, since this amounts to
proving the extension
estimate for $u$ supported
where $r\approx 1$ so that
one concludes using
\eqref{g0}, the embedding
$L^{p}\hookrightarrow
L^{p,p'}$ and H\"older's
inequality for $L^p$ spaces
(since $p<q'$).
\end{proof}
\begin{proof}[Proof that parabolic
restriction $\Rightarrow$ conical
restriction] Here we prove the
conclusions in Theorem \ref{t1},
assuming {\it ii}) or {\it ii})'.\par
 We
introduce orthogonal
coordinates
\[
(a_1,\dots,a_{n-1},a_n,b)=
\left(\xi_1,\ldots,\xi_{n-1},\frac{\tau+\xi_n}
{\sqrt{2}},\frac{\tau-\xi_n}{\sqrt{2}}\right),
\]
 and we use the notation $a=(a',a_n)$,
  $a'=(a_1,\ldots,a_{n-1})$. With such coordinates the cone
 $\tau=|\xi|$ has equation
 \[
 b=\frac{|a'|^2}{2a_n},\quad
 a_n>0,
 \]
 and the Lorentz invariant measure becomes
 the pull-back under the projection $(a,b)\mapsto a$ of the measure
 $\frac{da}{\sqrt{2}a_n}$.
 Again we denote by $L^{\alpha,\beta}$ the
Lorentz spaces on the real
semi-axis $\R^+$ with respect
to the Lebesgue measure. Set
moreover $x=(x',x_n)$.\par
Suppose {\it ii})'. We have
\begin{align}
(u\,d\sigma_{cone})^\vee(x,t)&=\frac{1}{\sqrt{2}}
\int e^{2\pi i\left(x\cdot
a+\frac{|a'|^2}{2a_n}t\right)}
u(a)\,\frac{da}{a_n}\nonumber\\
&=\frac{1}{\sqrt{2}}\int
e^{2\pi ix_n a_n} \int
e^{2\pi i\left(x'\cdot
a'+\frac{|a'|^2}{2a_n}t\right)}
u(a)\,\frac{da'}{a_n}\,da_n\nonumber.\nonumber
\end{align}
Hence, by \eqref{au} and
\eqref{inter}
\begin{align}
\|(u\,d\sigma_{cone})^\vee\|_{L^{p'}}&=
\|(u\,d\sigma_{cone})^\vee\|_{L^{p'}_{x',t}
L^{p'}_{x_n}}
\nonumber\\
&\lesssim \|\int e^{2\pi
i\left(x'\cdot
a'+\frac{|a'|^2}{2a_n}t\right)}
u(a)\,\frac{da'}{a_n}\|_{{L^{p'}_{x',t}}
L^{p,p'}_{a_n}}
 \nonumber\\
&\lesssim \|\int e^{2\pi
i\left(x'\cdot
a'+\frac{|a'|^2}{2a_n}t\right)}
u(a)\,\frac{da'}{a_n}\|_{{L^{p,p'}_{a_n}}
L^{p'}_{x',t}}. \nonumber
\end{align}
Now, changing variables and
the hypothesis give
\begin{align}
\|\int e^{2\pi i\left(x'\cdot
a'+\frac{|a'|^2}{2a_n}t\right)}
u(a)\,\frac{da'}{a_n}\|_{L^{p'}_{x',t}}&=
 a_n^{\frac{1}{p'}}\|\int
e^{2\pi i\left(x'\cdot
a'+\frac{|a'|^2}{2}t\right)}
u(a)\,\frac{da'}{a_n}\|_{L^{p'}_{x',t}}\nonumber\\
 &\lesssim
 a_n^{\frac{1}{p'}-1}\|{u(\cdot,a_n)}\|_{L^{q'}}.
 \nonumber
 \end{align}
We deduce
 \begin{align}
 \|(u\,d\sigma_{cone})^\vee\|_{L^{p'}}
 &\lesssim
 \| a_n^{\frac{1}{p'}-1}\|u(\cdot,a_n)\|
 _{L^{q'}}
 \|_{L^{p,p'}}\label{g02}\\
 &=\|a_n^{\frac{1}{q'}-\frac{1}{p}}
 a_n^{-\frac{1}{q'}}\|u(\cdot,a_n)\|
 _{L^{q'}}
 \|_{L^{p,p'}}\nonumber\\
 &\lesssim
 \|a_n^{\frac{1}{q'}-\frac{1}{p}}
 \cdot a_n^{-\frac{1}{q'}}\|u(\cdot,a_n)
 \|_{L^{q'}}\|_{L^{p,q'}},
 \nonumber
 \end{align}
for $p'>q'$.  Then one
concludes by applying
H\"older's
   inequality \eqref{holder}, since $
   L^{\frac{qp'}{p'-q},\infty}\cdot
   L^{q',q'}\hookrightarrow L^{p,q'}$.\par
   Let us now assume {\it
   ii}). By symmetry and the
   triangle inequality
   we can take $u$ supported
   in the sector
   $1\lesssim|a'|\leq
   a_n\lesssim 1$.
Then one can obtain the
desired estimate by using
\eqref{g02}, the embedding
$L^{p}\hookrightarrow
L^{p,p'}$ and H\"older's
inequality for $L^p$ spaces
(since $p<q'$).
\end{proof}
\begin{proof}[Proof that hyperbolic
restriction $\Rightarrow$ conical
restriction] Now we prove the
conclusions of Theorem \ref{t1} when
{\it iii}) or {\it iii})' are
satisfied.\par Denote by
$L^{\alpha,\beta}$ the Lorentz spaces
on $\R$ with the Lebesgue measure, and
set $x=(x',x_n)$.\par
 Assume first
{\it iii})'. We write the Fourier
extension operator as
\begin{align}
(u\,d\sigma_{cone})^\vee(x,t)&=\int
e^{2\pi
i\left(x\cdot\xi+t|\xi|\right)}
u(\xi)\,\frac{d\xi}{|\xi|}\nonumber\\
&=\int e^{2\pi ix_n \xi_n}
\int e^{2\pi i\left(x'\cdot
\xi'+t|\xi|\right)}
u(\xi)\,\frac{d\xi'}{|\xi|}\,d\xi_n\nonumber.\nonumber
\end{align}
By applying \eqref{au} and
\eqref{inter}
\begin{align}
\|(u\,d\sigma_{cone})^\vee\|_{L^{p'}}&=
\|(u\,d\sigma_{cone})^\vee\|_{L^{p'}_{x',t}
L^{p'}_{x_n}}
\nonumber\\
&\lesssim \|\int e^{2\pi
i\left(x'\cdot
\xi'+t|\xi|\right)}
u(\xi)\,\frac{d\xi'}{|\xi|}\|_{{L^{p'}_{x',t}}L^{p,p'}_
{\xi_n}}
 \nonumber\\
&\lesssim\|\int e^{2\pi
i\left(x'\cdot
\xi'+t|\xi|\right)}
u(\xi)\,\frac{d\xi'}{|\xi|}\|_{{L^{p,p'}_
{\xi_n}}L^{p'}_{x',t}}.
\nonumber
\end{align}
Now, a change of variables and the
hypothesis give
\begin{align}
\|\int &e^{2\pi
i\left(x'\cdot
\xi'+t|\xi|\right)}
u(\xi)\,\frac{d\xi'}{|\xi|}\|_{L^{p'}_{x',t}}
\\
&=
 |\xi_n|^{n-2-\frac{n}{p'}}\|\int
e^{2\pi i\left(x'\cdot
\xi'+t\sqrt{1+|\xi'|^2}
\right)}
u(\xi_n\xi',\xi_n)\frac{d\xi'}
{\sqrt{1+|\xi'|^2}}
\|_{L^{p'}_{x',t}}\nonumber\\
 &\lesssim|\xi_n|^{n-2-\frac{n}{p'}}
 \|u(\xi_n\xi',\xi_n)
 (1+|\xi'|^2)^{-\frac{1}{2q'}}
 \|_{L^{q'}_{\xi'}}.
 \nonumber
 \end{align}
 It follows that
 \begin{align}
 \|(u\,d\sigma_{cone})^\vee\|_{L^{p'}}&
 \lesssim
 \||\xi_n|^{n-2-\frac{n}{p'}}
 \|u(\xi_n\xi',\xi_n)
 (1+|\xi'|^2)^{-\frac{1}{2q'}}
 \|_{L^{q'}_{\xi'}}
 \|_{L^{p,p'}_{\xi_n}}\nonumber\\
 &=\||\xi_n|^{\frac{n-2}{q}-\frac{n}{p'}}
 \|u(\xi)|\xi|^{-\frac{1}{q'}}\|_{L^{q'}_{\xi'}}
 \|_{L^{p,p'}_{\xi_n}}\label{g03}.
 \end{align}
 Again one concludes by using the
 embedding $L^{p,q'}_{\xi_n}\hookrightarrow
 L^{p,p'}_{\xi_n}$ and H\"older's
 inequality
 \eqref{holder} (for the pair $p,q$ is scale invariant).\par
 Assume now {\it iii}). By symmetry and the triangle inequality
  we can take
  $u$ supported
where $|\xi'|\lesssim 1$ and
$\xi_n\approx1$. Hence the
desired conclusion follows
from \eqref{g03}, the
embedding
$L^{p}\hookrightarrow
L^{p,p'}$, and H\"older's
inequality for $L^p$ spaces
(since $p<q'$).
\end{proof}
\section{Proof of Theorem
\ref{t2}} We need the following result
on the normal form of a hypersurface
$S$ in $\R^n$, which is proved in
\cite{nicola} (Proposition 2.2).\par
For $P\in S$, denote by $\nu(P)$ the
number of principal curvatures which
vanish at $P$ (i.e. the dimension of
the kernel of the second fundamental
form at $P$).
\begin{pro}\label{prpr}
     Let $S$ be a hypersurface in
     $\R^n$, $P_0\in S$ and
     $\underline{\nu}:=\liminf\limits_{P\to{P_0}}\nu(P)
     \not=0,n-1$.
      There is an orthogonal system of
     coordinates $(\xi',\xi^{\prime\prime},\xi_n)$,
     $\xi'=(\xi_1,\ldots,\xi_{n-1-\underline{\nu}})$,
     $\xi^{\prime\prime}=(\xi_{n-\underline{\nu}},
     \ldots,\xi_{n-1})$
     with the origin at $P_0$
      such that, in a neighbourhood of
     $P_0$, $S$ is the graph of a function
     $\xi_n=\phi(\xi',\xi^{\prime\prime})$ of the type
  \begin{equation}\label{3star}
\phi(\xi',\xi^{\prime\prime})=\langle
M(\xi',\xi^{\prime\prime})\xi',\xi'\rangle,
\end{equation}
where $M$ is a square matrix of size
 $n-1-\underline{\nu}$ with
smooth entries.
   \end{pro}
We now prove a finer result for a
surface in $\R^3$ with Gaussian
curvature identically zero near a point
of type $\geq k$.
\begin{pro}\label{nf}
Let $S$ be a surface in
$\R^3$ and $P_0\in S$ a point
of type $\geq k$. Suppose
that the Gaussian curvature
of $S$ vanishes identically
near $P_0$. Then there is an
orthogonal system of
coordinates
$(\xi_1,\xi_2,\xi_3)$ with
the origin at $P_0$ such
that, in a neighbourhood of
$P_0$, $S$ is the graph of a
function
$\xi_3=f(\xi_1,\xi_2)$ of the
type
\begin{equation}\label{0a0}
f(\xi_1,\xi_2)=a(\xi_1,\xi_2)\xi_1^k,
\end{equation}
for some smooth function $a$
defined in an open
neighbourhood of $0$.
\end{pro}
We emphasise that the
transformation which brings
$S$ into the desired form is
an orthogonal one, and not
merely smooth. This will be
essential for applications to
the restriction problem.
Incidentally, we also see
that the notion of ``point of
type $k$" propagates along a
segment containing $P_0$ as
interior point (in
particular, the set of points
of type $k$ does not have
isolated points).
\begin{proof}[Proof of Proposition
\ref{nf}] The proof uses
induction on $k$. The
statement is true for $k=2$.
This follows from Proposition
\ref{prpr} with $n=3$, if
$\underline{\nu}=1$, whereas
if $\underline{\nu}=2$ then a
neighbourhood of $P_0$ lies
on a plane, and the result is
trivial. One could
 also obtain the result for $k=2$
 as a consequence of Corollary 6 of
 \cite{spivak}, page 359,
 if $P_0$ is of
type $2$ and Corollary 7 of
\cite{spivak}, page 361, if $P_0$ is of
type $>2$.\par Suppose then that the
statement is true with $k-1$ in place
of $k$ and let $P_0$ be a point of type
$\geq k\geq3$. By the inductive
hypothesis there are orthogonal
coordinates $(\xi_1,\xi_2,\xi_3)$ for
which $S$ coincides, near the origin,
with the graph of a function
\begin{equation}\label{a0}
f(\xi_1,\xi_2)=a(\xi_1,\xi_2)\xi_1^{k-1}.
\end{equation}
Observe that the hypothesis
on the Gaussian curvature can
be expressed by the equation
\begin{equation}\label{a1}
f_{\xi_1\xi_1}f_{\xi_2\xi_2}=
f_{\xi_1\xi_2}^2.
\end{equation}
After substituting in
\eqref{a1} the expression for
$f$ given in \eqref{a0} we
divide by $\xi_1^{2k-4}$ and
let $\xi_1\to0$. Upon setting
$\phi(t)=a(0,t)$ we find the
following (singular) Cauchy
problem
\begin{equation}\label{a2}
\phi
\phi^{\prime\prime}=
\frac{k-1}{k-2}(\phi^\prime)^2,\quad
\phi(0)=0.
\end{equation}
The initial condition in
\eqref{a2} comes from the
fact that $P_0$ is of type
$\geq k$ and $\nabla f(0)=0$,
so that $\partial^\alpha
f(0)=0$ for every
$\alpha\in\mathbb{Z}^2_+$,
$|\alpha|\leq k-1$.\par To
finish the proof it suffices
to verify that the only
solution to \eqref{a2} is the
trivial one: $\phi(t)=0$ for
every $t$. In fact, a Taylor
expansion of $a(\xi_1,\xi_2)$
at $\xi_1=0$ then gives
$a(\xi_1,\xi_2)=b(\xi_1,\xi_2)\xi_1$
for some smooth $b$, and
therefore
$f(\xi_1,\xi_2)=b(\xi_1,\xi_2)\xi_1^k$.\\
To this end we observe that
the maximal non-constant
solutions to the equation in
\eqref{a2} (in the region
where the Cauchy
wellposedness theorem
applies) are of the form
\[
\phi(t)=\pm(At+B)^{\frac{1}{1-\alpha}},\
\alpha=\frac{k-1}{k-2},\
A\not=0,
\] defined on
$(-B/A,+\infty)$ if $A>0$, or
$(-\infty,-B/A)$ if $A<0$. At
any rate, since $\alpha>1$,
they blow up at $t=-B/A$ and
an elementary continuity
argument then shows that
there is no solution
$\phi\not\equiv0$ with
$\phi(0)=0$.
\end{proof}
\begin{oss}\label{osservazioneaggiunta}\rm
We point out that useful
normal forms were obtained by
Schulz \cite{schulz} for
convex hypersurfaces $S$ of
finite type, in the sense
(different from that in the
present paper) that $S$ has
no tangents of infinite
order. In particular we see
that this condition is never
satisfied here, because in
\eqref{0a0} we have
$f(0,\xi_2)\equiv0$. Moreover
it is worth noting that the
normal forms in \cite{schulz}
are expressed in terms of a
Taylor expansion at a given
point $P_0$, whereas here we
deal with the geometry of $S$
in a whole neighbourhood of
$P_0$.
\end{oss}
We also recall the following
result by Sogge \cite{sogge}
(see also \cite{stein}, page
418).
\begin{teo}\label{sog} Let
$\psi$ be a smooth function
on an interval $[-a,a]$, with
$\psi^{(j)}(0)=0$ for $1\leq
j\leq k-1$ and
$\psi^{(k)}(0)\not=0$.
  Then there exist constants $0<\delta<a$
  and $C>0$ such that
  \[
  \|\int_{-\delta}^\delta e^{2\pi
  i (tx_1+\psi(t)x_2)}
  g(t)\,dt\|_{L^{p'}(\R^2)}\leq
  C \|g\|_{
  L^{q'}(-\delta,\delta)},
  \]
  for every $p'>4$, $p'\geq
k+2$, $p'\geq (k+1)q$, and
  $g\in
  L^{q'}(-\delta,\delta)$.
  Moreover the constants $\delta$ and $C$
   depend only on
   $a$,$p$,$q$,$k$, on upper
   bounds for finitely many
   derivatives of $\psi$ on
   $[-a,a]$ and a lower
   bound for
   $\psi^{(k)}(0)$.
  \end{teo}
The remark on the uniformity
of the constants $\delta$ and
$C$ did not appear in the
statement of \cite{sogge},
but it easily follows from
the proof of that result.
Also, the case $k=2$ was not
explicitly considered there
(because already treated in
\cite{zigmund,fefferman}),
but in any case it also
follows from the same proof,
under the additional (and
necessary) condition $p'>4$.
\begin{proof}[Proof of
Theorem \ref{t2}] We consider
the orthogonal system of
coordinates given in
Proposition \ref{nf}. Hence,
near the point $P_0$ (which
now coincides with the
origin), $S$ is the graph of
a function
$\xi_3=f(\xi_1,\xi_2)$ of the
form \eqref{0a0}. Let
$S_0=\{(\xi_1,\xi_2,f(\xi_1,\xi_2))\in\R^3:
|\xi_1|\leq \delta,
|\xi_2|\leq\delta\}$, where
$\delta>0$ is a small
constant that will be chosen
later on.\par Let $u$ be any
smooth function on $S$
supported in $S_0$. As usual
we will think of $u$ as a
function of the variables
$\xi_1,\xi_2$. Then we have
\begin{multline}
(ud\sigma)^\vee(x_1,x_2,x_3)\\
=\int e^{2\pi
ix_2\xi_2}\left(\int e^{2\pi
i
(x_1\xi_1+x_3f(\xi_1,\xi_2))}
u(\xi_1,\xi_2)\phi(\xi_1,\xi_2)\,d\xi_1\right)
\,d\xi_2,
\end{multline}
where
$\phi(\xi_1,\xi_2)=(1+|\nabla
f(\xi_1,\xi_2)|^2)^{1/2}$.
\par
Since $p'>2$, by the
Hausdorff-Young inequality we
have
\begin{align}
\|(ud\sigma)^\vee\|_{L^{p'}(\R^3)}&\lesssim
\| \int e^{2\pi i
(x_1\xi_1+x_3f(\xi_1,\xi_2))}
u(\xi_1,\xi_2)\phi(\xi_1,\xi_2)\,d\xi_1
\|_{L^{p'}_{x_1,x_3}
L^p_{\xi_2}}\nonumber\\
&\leq \|\int e^{2\pi i
(x_1\xi_1+x_3f(\xi_1,\xi_2))}
u(\xi_1,\xi_2)\phi(\xi_1,\xi_2)\,d\xi_1
\|_{L^p_{\xi_2}L^{p'}_{x_1,x_3}}.\label{a7}
\end{align}
Now we apply Theorem
\ref{sog} with
$\psi(t)=f(t,\xi_2)=a(t,\xi_2)t^k$.
We see that, if $\delta$ is
small enough, the hypothesis
in Theorem \ref{sog} is
satisfied uniformly with
respect to the parameter
$\xi_2$. Hence we deduce
that, if $\delta$ is small
enough, the norm in $L^{p'}$
with respect to $x_1,x_3$ of
the integral on the
right-hand side of \eqref{a7}
is not greater than
\[
C\|u(\cdot,\xi_2)\phi(\cdot,\xi_2)
\|_{L^{q'}},
\]
uniformly with respect to
$\xi_2$. It follows that
\[
\|(ud\sigma)^\vee\|_{L^{p'}(\R^3)}\lesssim
\|u\|_{L^p_{\xi_2}
L^{q'}_{\xi_1}},
\]
which gives the desired
estimate, since
$L^{q'}(-\delta,\delta)\hookrightarrow
L^p(-\delta,\delta)$ because
$q'>p$.
\end{proof}
\begin{oss} \rm By combining the normal form
\eqref{0a0} for the function $f$ which
defines $S$ near $P_0$ with Knapp-type
scaling arguments (see e.g.
\cite{tomas2,wolff2}) one can see that
the condition $p'\geq (k+1)q$ is always
necessary for \eqref{pm} to hold under
the hypotheses of Theorem \ref{t2}.\\
Indeed, \eqref{pm} implies that
\begin{multline*}
\|\int e^{2\pi i
(x_1\xi_1+x_2\xi_2+x_3f(\xi_1,\xi_2))}
u(\lambda\xi_1,\lambda^\epsilon\xi_2)
\phi(\xi_1,\xi_2)\,d\xi_1\,
d\xi_2\|_{L^{p'}(\R^3)}\\
\lesssim
\|u(\lambda\cdot,\lambda^\epsilon\cdot)
\|_{L^{q'}(\R^2)}
\end{multline*}
where $u\not\equiv0$ is a
fixed test function,
$\lambda$ is a large
parameter and $\epsilon>0$.
As above,
$\phi(\xi_1,\xi_2)=(1+|\nabla
f(\xi_1,\xi_2)|^2)^{1/2}$.\\
Changing variables gives
\begin{multline}\label{lall}
\|\int e^{2\pi i
(x_1\xi_1+x_2\xi_2+\lambda^k
x_3f(\xi_1/\lambda,\xi_2/\lambda^\epsilon))}
u(\xi_1,\xi_2)
\phi(\xi_1/\lambda,\xi_2/\lambda^\epsilon)\,d\xi_1\,
d\xi_2\|_{L^{p'}(\R^3)}\\
\lesssim
\lambda^{\frac{1+\epsilon}{q}-
\frac{1+k+\epsilon}{p'}}\|u
\|_{L^{q'}(\R^2)}
\end{multline}
 If one
assume, by contradiction,
 $p'<(k+1)q$, then there is $\epsilon>0$
  such that the exponent of $\lambda$
  on the right-hand side of \eqref{lall} is negative.
Hence, since $\lambda^k
f(\xi_1/\lambda,\xi_2/\lambda^\epsilon)\to
a(0,0)\xi_1^k$ as
$\lambda\to\infty$,
  by applying dominated convergence
   to the integral in $\xi$
   and the
   Fatou lemma to the integral in $x$ we
   obtain
\[
\|\int e^{2\pi i
(x_1\xi_1+x_2\xi_2+
x_3a(0,0)\xi_1^k)}
u(\xi_1,\xi_2)d\xi_1\,d\xi_2\|_{L^{p'}(\R^3)}=0,
\]
which is a contradiction
because of the uniqueness of
the Fourier
transform.\end{oss}

\end{document}